\documentclass[journal]{IEEEtran}
\usepackage{KaistDefs}
\allowdisplaybreaks

\begin{document}

\title{Model Predictive Tracking Control for Invariant Systems on Matrix Lie Groups via Stable Embedding into Euclidean Spaces}

\author{Dong Eui Chang, Karmvir Singh Phogat, and Jongeun Choi,~\IEEEmembership{Senior Member,~IEEE}
\thanks{Dong Eui Chang and Karmvir Singh Phogat are with the School of Electrical Engineering, KAIST, Daejeon, 34141, Korea.
        \tt{ dechang@kaist.ac.kr, karmvir.p@gmail.com
        }}
\thanks{Jongeun Choi is with the School of Mechanical Engineering, Yonsei University, 
           Seoul 03722, Korea.
	\tt{jongeunchoi@yonsei.ac.kr}}        

\thanks{This research has been in part supported by KAIST under grants G04170001 and N11180231 and   by   the Mid-career Research Program through the National Research Foundation of Korea  funded by the Ministry of Science and ICT under grant NRF-2018R1A2B6008063. (Corresponding Author: K.S. Phogat, Co-corresponding Author: J. Choi). (To be) Published in IEEE Transactions on Automatic Control. doi: 10.1109/TAC.2019.2946231}
}
% \markboth{IEEE Transactions on Automatic Control}%
% {}
%{Shell \MakeLowercase{\textit{et al.}}: Bare Demo of IEEEtran.cls for Journals}
\maketitle
%\pagenumbering{gobble} % To remove the page numbering

\begin{abstract}
For controller design for systems  on manifolds embedded in  Euclidean space, it is convenient to utilize a theory that  requires a single global coordinate system on the ambient Euclidean space   rather than multiple local charts on the manifold or  coordinate-free tools from differential geometry. In this article, we apply such a theory  to design model predictive tracking  controllers for systems whose dynamics evolve on manifolds and illustrate its efficacy with the fully actuated rigid body attitude control system. 
\end{abstract}


% \fi
% \IEEEpeerreviewmaketitle

\section{Introduction}

\IEEEPARstart{M}{odel} predictive control (MPC), which requires solving a constrained finite time-horizon optimal control problem, has been initially utilized mostly in slow  process industries \cite{borrelli2017predictive}. In contrast to conventional control schemes, MPC is prevalent in safety critical systems due to its ability to handle state and control constraints for large-scale systems \cite{borrelli,rawlings}. Due to the rise in computational power  and sophisticated algorithms, several successful MPC  implementations have recently been reported in various applications with fast dynamics, including autonomous vehicles \cite{Lee2013,huiping2017} and power electronics \cite{stellato}. Obviously, MPC designing strategies for continuous-time systems require linearized discrete-time systems, accounting for the system dynamics. Because linearization and discretization of the system dynamics are relatively daunting tasks on manifolds as compared to Euclidean spaces, designing MPC on manifolds is a nontrivial matter. First, a manifold cannot be entirely covered by one local coordinate chart unless it is diffeomorphic to a Euclidean space. As a result, one needs to  carry out coordinate changes when the system trajectory traverses through multiple charts. Second, linearization and discretization of system dynamics are both local approximations, so these procedures require use of local charts as well. In general, coordinate change is an expensive operation in terms of computation time, and it may introduce fairly large discontinuities to the dynamics due to switchings of local cost functions for MPC that are defined chartwise on the manifold. As many mechanical systems such as aerial vehicles and robotic systems evolve on manifolds, it is essential to have a theory that does not require switching of charts or unconventional tools from geometric control theory.  

 In this article, we propose a model predictive controller design for systems on manifolds by employing a stable embedding technique, which is summarized as follows: Given a control system on a manifold $M$, first embed the manifold into the Euclidean space $\mathbb R^n$ and extend the system dynamics stably to the ambient Euclidean space, i.e., the system dynamics  are extended on \m{\R^n} in such a way that the manifold \m{M} becomes an invariant attractor of the modified dynamics defined on \m{\R^n}. Since the system dynamics are now defined in $\mathbb R^n$, we can  carry out linearization in one single global Euclidean coordinate set for  $\mathbb R^n$ and then discretization of the linearized system in $\mathbb R^n$. The stable embedding technique increases the dimension of the system, and therefore MPC for the extended dynamics may be computationally more expensive than the dynamics in minimal coordinates; however, it simplifies linearization and discretization of the system dynamics to a large extent. This approach was successfully applied for linear stabilizing/tracking controller design \cite{Ch18} and structure-preserving numerical integration \cite{ChJiPe16}.

	 Recent attempts on MPC on  manifolds may be found in \cite{Kalabic2017, Andrew15b}, which require implicit representation of the system dynamics or explicit constraints in the optimization to preserve the manifold structure of \m{\SO}. In addition, these schemes require switching of charts as the local control law, which is needed for stability, is defined in charts. In contrast to these conventional schemes, we take the aforementioned stable embedding approach to design MPC for systems defined on manifolds. In our study, we consider a class of systems defined by fully actuated left invariant vector fields on matrix Lie groups and stably embed the system dynamics in Euclidean spaces. Subsequently, to track a reference trajectory, time-varying tracking error dynamics are defined in the ambient Euclidean space; those are linearized along the reference and simplified using symmetry invariance, both in one single global  coordinate system on the ambient Euclidean space, and the local stabilizability of the original nonlinear tracking error dynamics to zero  is then readily established. For applying MPC, the  error dynamics are linearized and discretized, and the stabilizability of the discrete-time linear error dynamics is also proven, for which a  fundamental sufficient condition is derived  in an inequality form that involves the two parameters: the  discretization time step and the transversal stability  parameter that is introduced in the process of stable embedding. Later, an MPC law is designed for the discrete-time linear error dynamics and is applied to the original nonlinear system. It is worth mentioning at this juncture that, to the best of the authors' knowledge, the issue of establishing exponential stability of the time-varying system dynamics under the synthesized MPC control law remains open. Some results on the stability of the sampled data systems may be found in \cite{nesic1999,nesic2004}. To demonstrate the efficacy of the proposed MPC technique, we design a tracking controller for spacecraft attitude control dynamics and conduct numerical studies for various real-time scenarios, such as reference tracking under tight control constraints and noisy measurements, to illustrate the validity of the designed controller. Numerical simulations show that the MPC tracking controller designed using the discrete-time linear system in the Euclidean space is robust to unmodeled disturbances and provides good tracking performance when applied to the actual nonlinear system.    
 
The structure of this article is as follows: We establish stabilizability of the tracking error dynamics for left invariant systems on matrix Lie groups and discuss the design procedure for the MPC controller in Section \ref{sec:MpcLie} . Section \ref{sec:RigidBody} is devoted to design of an MPC tracking controller for a rigid body attitude control system. Numerical studies of the designed tracking controller for attitude dynamics are in Section \ref{sec:Simulations}, followed by our conclusions in Section \ref{sec:Conclusion}. 
 
\section{MPC on Matrix Lie groups} \label{sec:MpcLie}
Let \m{\LieGrp} be a matrix Lie group of dimension $m$ with \m{\gI} as the group identity and \m{\LieAlg} be the Lie algebra of the Lie group \m{\LieGrp}.  Suppose that a controlled system dynamics on the matrix Lie group is given by a left invariant vector field: 
\begin{equation}\label{eq:DynLie}
\Sigma: \quad
\begin{cases}
\dot{\gst} = \gst \vst, \\
\dot{\vst} = f(\vst,u), 
\end{cases}
\end{equation}
where \m{\gst \in \LieGrp,\, \vst \in \LieAlg} and \m{u \in \R^m}. It is assumed that $\frac{\partial f}{\partial u}(\xi,u)$ is an invertible linear map from $\R^m$ to $\LieAlg$ for each $(\xi,u) \in \LieAlg \times \R^{m}$.
Suppose that the matrix Lie group \m{\LieGrp} is embedded into a Euclidean space \m{\Eg}. The vector space $\mathbb R^{n\times n}$ is split into two orthogonal subspaces \m{\LieAlg} and \m{\prp{\LieAlg}} such that \m{\Eg = \LieAlg \oplus \prp{\LieAlg}}, where \m{\prp{\LieAlg}} is the orthogonal complement of \m{\LieAlg} in \m{\Eg} under the Euclidean metric defined by $\langle A, B\rangle = \operatorname{trace}(A^TB)$ for all $A,B\in \mathbb R^{n\times n}$. In the subsequent discussion, we refer to \m{\LieAlg} as the parallel direction and \m{\prp{\LieAlg}} as the transversal direction, and we define orthogonal projection maps from the ambient space \m{\Eg} to \m{\LieAlg} and \m{\prp{\LieAlg}} as
\[
 \Eg \ni v \mapsto \prl{v} \in \LieAlg, \quad \Eg \ni v \mapsto \prp{v} \in \prp{\LieAlg}. 
\]
A detailed discussion on left invariant systems and the differential geometric tools employed in this article may be found in \cite{abraham,sastry,bloch}.  
 Let us turn to stably embed the system dynamics \eqref{eq:DynLie} into \m{\Eg} considering the following assumption: 
\begin{assumption}
\label{asm:V}
There exists a $C^2$ function 
\[
\Eg \ni x \mapsto V(x) \geq 0 \in \R
\]
with the following properties: 

\begin{enumerate}[label=\textup{(A-\roman*)}, leftmargin=*, widest=b, align=left]

\item \label{asm:1} \m{ V^{-1}(0)= \LieGrp}.
\item \label{asm:2} \m{V(\st \gst)= V(\st)} for all \m{\st \in \Eg, \gst \in \LieGrp.} 
\item \label{asm:3} \m{\nabla^2 V(\gI)} is positive definite in the transversal direction, i.e., \m{\nabla^2 V(\gI)\cdot (y,y)>0} for all \m{y \in \prp{\LieAlg}\backslash\{0\}}.
\end{enumerate} 
\end{assumption}
 Since $V$ attains its minimum value 0 at each point in $\LieGrp$, $\nabla V$ vanishes on $\LieGrp$. Hence, the system dynamics \eqref{eq:DynLie}  can be extended  to   the ambient Euclidean space \m{\Eg \times \LieAlg} as
\begin{equation} \label{eq:DynEg}
\tilde{\Sigma}: \quad
\begin{cases}
\dot{\st} = \st \vst - \alpha \nabla V(\st), \\
\dot{\vst} = f(\vst,u), 
\end{cases}
\end{equation}
where \m{\alpha>0,} \m{ x \in \Eg,\, \vst \in \LieAlg} and \m{u \in \R^m}. Let 
\begin{align}\label{eq:Ref}
\R \ni t \mapsto (\Rtraj{\gst}(t),\Rtraj{\vst}(t)) \in \LieGrp \times \LieAlg
\end{align}
be a reference state trajectory and \m{\R \ni t \mapsto \Rtraj{u}(t) \in \R^m} be the corresponding control trajectory of the system dynamics \eqref{eq:DynLie}. 
\begin{assumption}\label{asm:g0}
There exist constants \m{\gub \geq \glb >0} such that the reference trajectory 
\[
\R \ni t\mapsto \Rtraj{\gst}(t) \in \LieGrp
\]
satisfies 
\[
\glb I \mleq \Rtraj{\gst}(t)\bigl(\Rtraj{\gst}(t)\bigr)^\top \mleq \gub I \quad \text{for all } t.
\]
\end{assumption}
The tracking error trajectory, defined by 
\begin{align*}
\R \ni t & \mapsto \left(\erg(t), \era (t) \right)\\ 
& \Let \left(\st(t)\Rtraj{\gst}^{-1}(t) - \gI, \vst(t) - \Rtraj{\vst}(t)\right) \in \Eg\times\LieAlg
\end{align*}
such that \m{\left(\erg(t), \era(t) \right)=0} for all \m{t} ensures, that the system dynamics \eqref{eq:DynEg} is tracking the reference trajectory, i.e., \m{\st(t)=\Rtraj{\gst}(t), \vst(t) = \Rtraj{\vst}(t)} for all \m{t}. Therefore, the reference tracking problem is translated to stabilization of the error dynamics to zero. The error dynamics for a given reference trajectory is given as
\begin{equation}\label{eq:DynError}
\er \Sigma: \quad
\begin{cases}
\dot{\erg} = \big(\Rtraj{\gst} + \erg \Rtraj{\gst}\big)\era\Rtraj{\gst}^{-1} - \alpha \nabla V(\Rtraj{\gst} + \erg \Rtraj{\gst})\Rtraj{\gst}^{-1},\\
\dot{\era} = f(\era + \Rtraj{\vst},u) - f(\Rtraj{\vst},\Rtraj{u}).
\end{cases}
\end{equation}
To design a linear MPC controller, let us linearize the error dynamics \eqref{eq:DynError} around zero. The linearized error dynamics around zero is given by
\begin{equation}\label{eq:DynErrorLin}
\er \Sigma^\ell: \quad
\begin{cases}
\dot{\erg} = \Rtraj{\gst}\era\Rtraj{\gst}^{-1} - \alpha \bigl(\nabla^2 V(\Rtraj{\gst})\cdot(\erg \Rtraj{\gst}) \bigr) \Rtraj{\gst}^{-1},\\
\dot{\era} = \frac{\partial f}{\partial \vst} (\Rtraj{\vst},\Rtraj{u}) \era + \frac{\partial f}{\partial u} (\Rtraj{\vst},\Rtraj{u}) \eru,
\end{cases}
\end{equation}
where \m{\eru \Let u - \Rtraj{u}}.
Before simplifying the error dynamics \m{\dot{\erg}} in \eqref{eq:DynErrorLin} and splitting it further along the parallel and the transversal direction to gain more geometric insight, let us discuss some key properties associated with the function \m{V}:    
\begin{lemma}\label{lem:V_property} 
Under Assumption \ref{asm:V},  the following hold:
\begin{enumerate}[label=\textup{(\alph*)}, leftmargin=*, widest=b, align=left]
\item \m{\nabla^2 V(I) \cdot\prl{v} = 0} for all \m{\prl{v} \in \LieAlg}. \label{lm:1}
\item \m{\nabla^2 V(I) \cdot \prp{v} \in \prp{\LieAlg}} for all \m{\prp{v} \in \prp{\LieAlg}}. 
\item \m{\nabla^2 V(g)\cdot (v g) = \bigl(\nabla^2 V(\gI)\cdot v\bigr) \big(g^{-1}\big)^\top} for all \m{v \in \Eg} and \m{g\in\LieGrp}.
\end{enumerate}
\end{lemma}
\begin{proof}
\begin{enumerate}
\item Note that \m{\nabla V (g) = 0} for all \m{g \in \LieGrp.} Therefore,
\begin{equation*}
\nabla^2 V(I) \cdot \prl{v} = \left.\frac{d}{ds}\right|_{s=0} \nabla V (\exp(s \prl{v})) = 0.
\end{equation*}

\item For any \m{\prp{v} \in \prp{\LieAlg}},  \m{\nabla^2 V(I)\cdot \prp{v} \in \prp{\LieAlg}} if and only if \m{\ip{\prl{v}}{\nabla^2 V(I)\cdot \prp{v}} =0 } for all \m{\prl{v} \in \LieAlg.} Therefore, using the fact that \m{\nabla^2 V(I)} is symmetric and then applying Lemma \ref{lem:V_property}\ref{lm:1} leads to the following:
\[
\ip{\prl{v}}{\nabla^2 V(I)\cdot\prp{v}} = \ip{\nabla^2 V(I)\cdot \prl{v} }{\prp{v}} = 0  
\]
for all  \m{\prl{v} \in \LieAlg.}
\item For arbitrary vectors \m{w,v \in \Eg}, 
\begin{align*}
\ip{w}{\nabla^2 V(g)\cdot (v g)} & = \left.\frac{d}{dt}\frac{d}{ds}\right|_{t=s=0} V(g+t v g+s w) \\
& = \left.\frac{d}{dt}\frac{d}{ds}\right|_{t=s=0} V(\gI+t v+s w g^{-1})\\
& = \ip{w g^{-1} }{\nabla^2 V(\gI)\cdot v} \\
& = \ip{w }{\bigl(\nabla^2 V(\gI)\cdot v\bigr)(g^{-1})^\top}.
\end{align*}
Therefore, we conclude that
\[
\nabla^2 V(g)\cdot (v g) = \bigl(\nabla^2 V(\gI)\cdot v\bigr) (g^{-1})^\top.
\]
\end{enumerate}
\end{proof}

Employing the properties discussed in Lemma \ref{lem:V_property}, it is straightforward to show that the linearized error dynamics in \eqref{eq:DynErrorLin} is transformed to the following:
\begin{subequations} \label{eq:SplitDyn}
\begin{align}
\prp{\dot{\erg}} &= - \alpha \prp{\Bigl(\bigl(\nabla^2 V(\gI)\cdot\prp{\erg}\bigr) \bigl(\Rtraj{\gst}\Rtraj{\gst}^\top\bigr)^{-1}\Bigr)}, \label{eq:SplitDynPrp}\\
\prl{\dot{\erg}} &= \Rtraj{\gst}\era\Rtraj{\gst}^{-1} - \alpha \prl{\Bigl(\bigl(\nabla^2 V(\gI)\cdot\prp{\erg}\bigr) \bigl(\Rtraj{\gst}\Rtraj{\gst}^\top\bigr)^{-1}\Bigr)}, \label{eq:SplitDynPrl}\\
\dot{\era} &= \frac{\partial f}{\partial \vst} (\Rtraj{\vst},\Rtraj{u}) \era + \frac{\partial f}{\partial u} (\Rtraj{\vst},\Rtraj{u}) \eru, \label{eq:SplitDynXi}
\end{align}
\end{subequations}
where the linear error \m{\erg} in \eqref{eq:DynErrorLin} has simplified and split into the transversal direction error \m{\R \ni t \mapsto \prp{\erg}(t) \in \prp{\LieAlg}} and  the parallel direction error \m{\R\ni t\mapsto \prl{\erg}(t) \in \LieAlg}.   
  
\begin{theorem}\label{thm:perp:dyn:stab}
Under Assumptions \ref{asm:V} and \ref{asm:g0},
the system dynamics \eqref{eq:SplitDynPrp} is exponentially stable to zero.
\end{theorem}
\begin{proof}
Let us define a candidate Lyapunov function 
\[
\prp{\LieAlg} \ni \eta \mapsto \mathcal{V}(\eta) \Let \frac{1}{2} \ip{\nabla^2 V(\gI)\cdot \eta}{\eta} \in \R.
\]
Then, using the properties of \m{V} discussed in Lemma \ref{lem:V_property}, we obtain:  
\begin{align*}
& \frac{d\mathcal{V}}{dt}\bigl(\prp{\erg}\bigr)  = \ip{\frac{\partial \mathcal{V}}{\partial \eta} (\prp{\erg})}{\prp{\dot{\erg}}} \\
& = - \alpha \ip{\nabla^2 V(\gI)\cdot \prp{\erg}}{\prp{\Bigl(\bigl(\nabla^2 V(\gI)\cdot\prp{\erg}\bigr) \bigl(\Rtraj{\gst}\Rtraj{\gst}^\top\bigr)^{-1}\Bigr)}} \\
& = - \alpha \ip{\nabla^2 V(\gI)\cdot \prp{\erg}}{\bigl(\nabla^2 V(\gI)\cdot\prp{\erg}\bigr) \bigl(\Rtraj{\gst}\Rtraj{\gst}^{\top}\bigr)^{-1}} \\
& \leq - \frac{\alpha}{\gub} \ip{\nabla^2 V(\gI)\cdot \prp{\erg}}{\nabla^2 V(\gI)\cdot\prp{\erg}} \\
& \leq - \frac{2 \alpha \lambda_{\text{min}}}{\gub} \mathcal{V}\bigl(\prp{\erg}\bigr),
\end{align*}
where \m{\lambda_{\text{min}}} is the smallest eigenvalue of the operator \m{\nabla^2 V(\gI)} restricted to  $\mathfrak{g}^\perp$.
Therefore, 
\[
\mathcal{V}\big(\prp{\erg}(t)\big) \leq \exp\Big(- \frac{2 \alpha \lambda_{\text{min}}}{\gub} t \Big) \mathcal{V}\big(\prp{\erg}(0)\big),
\]
and by the definition of \m{\mathcal{V}}, we know that 
\[
\lambda_{\text{min}} \norm{\prp{\erg}}^2 \leq 2 \mathcal{V}(\prp{\erg}) \leq \lambda_{\text{max}} \norm{\prp{\erg}}^2, 
\]
where \m{\lambda_{\text{max}}} is the largest eigenvalue of the operator \m{\nabla^2 V(\gI)} restricted to  $\mathfrak{g}^\perp$.
So,
\[ 
\norm{\prp{\erg}(t)} \leq \sqrt{\frac{\lambda_{\text{max}}}{\lambda_{\text{min}}}} \exp\Big(- \frac{ \alpha \lambda_{\text{max}}}{\gub}t\Big) \norm{\prp{\erg}(0)}.
\]
This proves the assertion. 
\end{proof}

\begin{remark}
The proof of Theorem \ref{thm:perp:dyn:stab} uses the boundedness of $g_0(t)\bigl(\Rtraj{\gst}(t)\bigr)^\top$ from above, see Assumption \ref{asm:g0}. 
\end{remark}

The linearized error dynamics \eqref{eq:SplitDyn}  can be exponentially stabilized to zero by feedback. Consequently, the original nonlinear tracking error dynamics \eqref{eq:DynError} is exponentially stabilizable. 

\begin{theorem}\label{thm:Stab}
Suppose Assumptions \ref{asm:V} and \ref{asm:g0} hold.  Then, for any two matrices \m{K_p,K_d \in \Eg} such that the matrix 
\begin{align}\label{eq:ContStabM}
\begin{pmatrix}
0 &  I \\ K_p & K_d
\end{pmatrix}
\end{align}
is \textit{Hurwitz}, the PD-like controller
\begin{align}\label{eq:ContStabControl}
\eru = \Big(\frac{\partial f}{\partial u} (\Rtraj{\vst},\Rtraj{u})\Big)^{-1}\Big\{[\era,\Rtraj{\vst}] + Y  - \frac{\partial f}{\partial \vst} (\Rtraj{\vst},\Rtraj{u})\era \Big\} 
\end{align}
with
\begin{equation}\label{expr:Y}
Y \Let \Rtraj{\gst}^{-1} \big(-\dot{W} + K_p \prl{\erg} + K_d (\Rtraj{\gst} \era \Rtraj{\gst}^{-1} + W ) \big)\Rtraj{\gst}
\end{equation}
and
\[
W \Let  - \alpha \prl{\Bigl(\bigl(\nabla^2 V(\gI)\cdot \prp{\erg} \bigr)\bigl(\Rtraj{\gst}\Rtraj{\gst}^\top\bigr)^{-1}\Bigr)},
\]
stabilizes the controlled dynamics \eqref{eq:SplitDyn} exponentially to zero, where $\dot{W}$  in \eqref{expr:Y} can be expressed in terms of state variables using \eqref{eq:SplitDynPrp} and $\dot{g}_0 = g_0 \xi_0$.  Furthermore, the controller $u = u_0 + \delta u$ exponentially stabilizes \eqref{eq:DynError} to zero.
\end{theorem}  
\begin{proof}
Since the exponential stability of the transversal dynamics \eqref{eq:SplitDynPrp} has been proved in Theorem \ref{thm:perp:dyn:stab}, the exponential stability of the subsystems \eqref{eq:SplitDynPrl} and \eqref{eq:SplitDynXi} with the control law \eqref{eq:ContStabControl} remains to be proved. Applying the controller \eqref{eq:ContStabControl} to subsystems \eqref{eq:SplitDynXi} transforms the subsystems \eqref{eq:SplitDynPrl} and \eqref{eq:SplitDynXi} to the following system of  differential equations
\begin{equation}
\begin{aligned}\label{eq:ContExpStab}
\begin{pmatrix}
\dot{\prl{\erg}} \\
\dot{\eraa}
\end{pmatrix}
 = 
\begin{pmatrix}
0 & I \\ K_p & K_d
\end{pmatrix}
\begin{pmatrix}
\prl{\erg} \\
\eraa
\end{pmatrix},
\end{aligned}
\end{equation}
where 
\begin{align}\label{def:ContErAux}
\eraa \Let \Rtraj{\gst} \era \Rtraj{\gst}^{-1} - \alpha \prl{\Bigl(\bigl(\nabla^2 V(\gI)\cdot \prp{\erg} \bigr)\bigl(\Rtraj{\gst}\Rtraj{\gst}^\top\bigr)^{-1}\Bigr)}.
\end{align}
Therefore, the linear system \eqref{eq:ContExpStab} is exponentially stable to zero if the matrix \eqref{eq:ContStabM} is {\it Hurwitz} stable.
Since \m{\prp{\erg}} and \m{\eraa} are exponentially stable, it follows from \eqref{def:ContErAux}, Theorem \ref{thm:perp:dyn:stab} and Assumption \ref{asm:g0}  that \m{\era} is also exponentially stable.  By the Lyapunov linearization method, this control law also exponentially stabilizes \eqref{eq:DynError} to zero.
This proves the assertion.  
\end{proof}

	To apply the MPC, let us discretize the linearized error dynamics \eqref{eq:SplitDyn}  using Euler's method as
\begin{subequations}\label{eq:DynDis}
\begin{align}
\dis[k+1]{\prp{\erg}} &= \dis{\prp{\erg}}-\steplength \alpha  \prp{\Bigl(\bigl(\nabla^2 V(\gI)\cdot \prp{\dis{\erg}} \bigr)\bigl(\dRtraj{\gst}\dRtraj{\gst}^\top\bigr)^{-1}\Bigr)}, \label{eq:DynDisPrp}\\ 
\dis[k+1]{\prl{\erg}} &= \dis{\prl{\erg}} + \steplength \dRtraj{\gst} \dis{\era}\dRtraj{\gst}^{-1}\nonumber \\
&\quad - \alpha\steplength \prl{\Bigl(\bigl(\nabla^2 V(\gI)\cdot \prp{\dis{\erg}} \bigr)\bigl(\dRtraj{\gst}\dRtraj{\gst}^\top\bigr)^{-1}\Bigr)},  \label{eq:DynDisPrl}\\
\dis[k+1]{\era} &= \dis{\era} + \steplength  \frac{\partial f}{\partial \vst} (\dRtraj{\vst},\dRtraj{u}) \dis{\era} + \steplength \frac{\partial f}{\partial u} (\dRtraj{\vst},\dRtraj{u})\dis{\eru},\label{eq:DynDisAlg}
\end{align}  
\end{subequations}
where $h$ is  a discretization step, and for a function  \m{\R \ni t \mapsto \Gamma(t) \in \Eg,\; \Gamma_k \Let \Gamma(kh).}  
The stability of the MPC for discrete-time systems requires stabilizability of these discrete-time systems \cite{rawlings93}. Therefore, it is crucial to prove stabilizability of the discrete-time dynamics \eqref{eq:DynDis}. First, we prove the exponentially stability of the subsystem \eqref{eq:DynDisPrp} for an appropriate choice of the discretization step \m{\steplength} and the parameter \m{\alpha}. Then, we proceed to the general case and establish stabilizability of the discrete-time dynamics \eqref{eq:DynDis}. In addition, Theorem \ref{thm:DisStabPrp} establishes a relation between the stabilizing parameter \m{\alpha} and the discretization step length \m{\steplength} that is crucial in implementation of MPC.  
\begin{theorem}\label{thm:DisStabPrp}
Suppose that Assumptions \ref{asm:V} and \ref{asm:g0} hold. Then, the system dynamics \eqref{eq:DynDisPrp} is exponentially stable to zero if the following holds:
\[
0 < \alpha \steplength < \frac{2\lambda_{\text{min}}\glb^2}{\lambda^2_{\text{max}}\gub},
\]
where \m{\lambda_{\text{min}}} and \m{\lambda_{\text{max}}} are the minimum and the maximum eigenvalues of the operator \m{\nabla^2 V(\gI)} restricted to  $\mathfrak{g}^\perp$, respectively. 
\end{theorem}
\begin{proof}
Let us define a candidate Lyapunov function 
\[
\prp{\LieAlg} \ni \eta \mapsto \mathcal{V}_d(\eta) \Let \ip{\nabla^2 V(\gI)\cdot \eta}{\eta} \in \R.
\]
Then, using the properties of \m{V} from Lemma \ref{lem:V_property} gives
\begin{align*}
& \mathcal{V}_d\bigl(\prp{\dis[k+1]{\erg}}\bigr)  = \ip{\nabla^2 V(\gI) \cdot \prp{\dis[k+1]{\erg}}}{\prp{\dis[k+1]{\erg}}} \\
& = \ip{\nabla^2 V(\gI)\cdot\prp{\dis{\erg}}}{\prp{\dis{\erg}}}\\
& -2 \alpha \steplength \ip{\nabla^2 V(\gI)\cdot \prp{\dis{\erg}}}{\bigl(\nabla^2 V(\gI)\cdot \prp{\dis{\erg}}\bigr)\bigl(\dRtraj{\gst} \dRtraj{\gst}^{\top}\bigr)^{-1}}\\  
& + \alpha^2 \steplength^2 \anorm{\bigl(\nabla^2 V(\gI)\cdot \prp{\dis{\erg}}\bigr)\bigl(\dRtraj{\gst} \dRtraj{\gst}^{\top}\bigr)^{-1}}_{\nabla^2V(\gI)}^{2}\\
& \leq \biggl(1- 2 \alpha \steplength \frac{\lambda_{\text{min}}}{\gub} +  \alpha^2 \steplength^2 \frac{\lambda^2_{\text{max}}}{\glb^2} \biggr)  \mathcal{V}_d\bigl(\prp{\dis{\erg}}\bigr) \\
& \leq \biggl\{ \Bigl(\alpha \steplength \frac{\lambda_{\text{max}}}{\glb}-1\Bigr)^2 + 2 \alpha \steplength \Bigl(\frac{\lambda_{\text{max}}}{\glb} - \frac{\lambda_{\text{min}}}{\gub}   \Bigr)\biggr\} \mathcal{V}_d\bigl(\prp{\dis{\erg}}\bigr). 
\end{align*}
Therefore, the system dynamics \eqref{eq:DynDisPrp} is exponentially stable if 
\[
0 \leq \Bigl(\alpha \steplength \frac{\lambda_{\text{max}}}{\glb}-1\Bigr)^2 + 2 \alpha \steplength \Bigl(\frac{\lambda_{\text{max}}}{\glb} - \frac{\lambda_{\text{min}}}{\gub}   \Bigr) < 1 
\] 
which leads to the following condition: 
\[
0 < \alpha \steplength < \frac{2\lambda_{\text{min}}\glb^2}{\lambda^2_{\text{max}}\gub}.
\]
This proves the assertion. 
\end{proof}

\begin{theorem}\label{thm:DisStab}
Suppose that Assumptions \ref{asm:V} and \ref{asm:g0} hold. Then, for any two matrices \m{K_p,K_d \in \Eg} such that the matrix 
\begin{align}\label{eq:StabMatrix}
\begin{pmatrix}
I & h I \\ K_p & K_d
\end{pmatrix}
\end{align}
is {\it Schur} stable, the controller
\begin{align}\label{eq:StabControl}
\dis{\eru} = \frac{1}{\steplength} \Big(\frac{\partial f}{\partial u} (\dRtraj{\vst},\dRtraj{u})\Big)^{-1}\Big\{ & \dis{Y} - \left (I+h \frac{\partial f}{\partial \vst} (\dRtraj{\vst},\dRtraj{u}) \right )\dis{\era} \Big\} 
\end{align}
where \[
\dis{Y} \Let \dRtraj[,k+1]{\gst}^{-1} \Big(K_p \dis{\prl{\erg}} + K_d\big(\dRtraj{\gst}\dis{\era}\dRtraj{\gst}^{-1} + \dis{W}\big) - \dis[k+1]{W}\Big) \dRtraj[,k+1]{\gst}
\]
and
\[
\dis{W} \Let  - \alpha \prl{\Bigl(\bigl(\nabla^2 V(\gI)\cdot \prp{\dis{\erg}} \bigr)\bigl(\dRtraj{\gst}\dRtraj{\gst}^\top\bigr)^{-1}\Bigr)}, 
\]
stabilizes the controlled dynamics \eqref{eq:DynDis} exponentially to zero. 
\end{theorem}  
\begin{proof}
Since the exponential stability of the transversal dynamics \eqref{eq:DynDisPrp} has been proved in Theorem \ref{thm:DisStabPrp}, the exponential stability of the subsystems \eqref{eq:DynDisPrl} and \eqref{eq:DynDisAlg} with the control law \eqref{eq:StabControl} remains to be proved. Applying the controller \eqref{eq:StabControl} to subsystems \eqref{eq:DynDisAlg} transforms the subsystems \eqref{eq:DynDisPrl} and \eqref{eq:DynDisAlg} to the following system of  difference equations
\begin{equation}
\begin{aligned}\label{eq:ExpStab}
\begin{pmatrix}
\dis[k+1]{\prl{\erg}} \\
\dis[k+1]{\eraa}
\end{pmatrix}
 = 
\begin{pmatrix}
I & h I \\ K_p & K_d
\end{pmatrix}
\begin{pmatrix}
\dis{\prl{\erg}} \\
\dis{\eraa}
\end{pmatrix}
\end{aligned}
\end{equation}
where 
\begin{align}\label{def:ErAux}
\dis{\eraa}\Let \dRtraj{\gst} \dis{\era}\dRtraj{\gst}^{-1} - \alpha \prl{\Bigl(\bigl(\nabla^2 V(\gI)\cdot \prp{\dis{\erg}} \bigr)\bigl(\dRtraj{\gst}\dRtraj{\gst}^\top\bigr)^{-1}\Bigr)}.
\end{align}
Therefore, the linear system \eqref{eq:ExpStab} is exponentially stable to zero if the matrix \eqref{eq:StabMatrix} is {\it Schur} stable.
Since \m{\dis[k]{\prl{\erg}}} and \m{\dis[k]\eraa} are exponentially stable, it follows from \eqref{def:ErAux}, Theorem \ref{thm:DisStabPrp} and Assumption \ref{asm:g0} that \m{\era} is also exponentially stable.
This proves the assertion.  
\end{proof}
Equipped with Theorem \ref{thm:DisStab}, we are in a position to design an MPC control law  for the system dynamics \eqref{eq:DynDis}. MPC computes a static state feedback control law at each time instant by solving a constrained finite horizon discrete-time optimal control problem. A typical optimal control problem for a horizon \m{N} is to minimize a performance objective 
\begin{equation}
\begin{aligned}\label{ocp:cost}
\J(\ero_{0:N},&\erm_{0:N},\eru_{0:N-1})  \Let \sum_{k=0}^{N-1}( \norm{\dis[k]{\ero}}_{Q_{\ero}}^2 + \norm{\dis[k]{\erm}}_{Q_{\erm}}^2) \\
& + \sum_{k=0}^{N-1} \norm{\dis[k]{\eru}}_{Q_{\eru}}^2  +  \norm{\dis[N]{\ero}}_{Q^f_{\ero}}^2 + \norm{\dis[N]{\erm}}_{Q^f_{\erm}}^2, 
\end{aligned}
\end{equation}
where \m{Q_{\ero},Q_{\erm},Q_{\eru},Q_{\ero}^f,Q_{\erm}^f \in \Eg} are positive semidefinite matrices, while satisfying the system dynamics \eqref{eq:DynDis} 
and the state and control constraints
\begin{equation}
\begin{aligned}\label{ocp:constraints}
&\dis{{\ero}} \in \cero \quad \text{for all } k=0,\ldots,N,\\
&\dis{\erm} \in \cerm \quad \text{for all } k=0,\ldots,N,\\
&\dis{\eru} \in \ceru \quad \text{for all } k=0,\ldots,N-1,
\end{aligned}
\end{equation}
where \m{\cero,\cerm} are the admissible state sets and \m{\ceru} is an admissible action set at each time instant \m{k}.

	Concisely, an optimal control \m{\dis[j|j]{\eru}} at the time instant \m{j} for a fixed given state \m{\big(\dis[j|j]{\ero},\dis[j|j]{\erm} \big)} is obtained by solving the following constrained discrete-time optimal control problem:
\begin{equation}\label{ocp:DynDis}
\begin{aligned}
&\minimize_{\{\dis[j+i|j]{\eru}\}_{i=0}^{N-1}} \J \big(\dis[j:j+N|j]{\ero},\dis[j:j+N|j]{\erm},\dis[j:j+N|j]{\eru}\big)\\
&\text{subject to}\\
&\quad\begin{cases}
\text{dynamics } \eqref{eq:DynDis}\;\; \text{ for } k=j|j,\ldots,j+N-1|j,\\
\text{constraints } \eqref{ocp:constraints}\;\; \text{ for } k=j|j,\ldots,j+N-1|j,\\
\big(\dis[j|j]{\ero},\dis[j|j]{\erm}\big) \text{ is fixed}.\\
\end{cases} 
\end{aligned} 
\end{equation}
Then, the control law
\[
u (t)= u_0 (t)+  \dis[j|j]{\eru}
\]
is applied to the system \eqref{eq:DynLie} for $t \in [jh,  (j+1)h[ $, where $j = 0, 1,2, \cdots$. 
\begin{remark}
Note that the system dynamics \eqref{eq:DynDis} is exponentially stabilizable. Therefore, we design an exponentially stabilizing MPC law  for the dynamics \eqref{eq:DynDis} in Euclidean spaces that in turn stabilizes Euler's approximation of the error dynamics \eqref{eq:DynError} exponentially to zero \cite{falcone2008}. However, to the best of the authors' knowledge, the issue of establishing exponential stability of the time-varying sampled data system \eqref{eq:DynError} under the synthesized MPC control law remains open. Some results on the stability of the sampled data systems may be found in \cite{nesic1999,nesic2004}. 
\end{remark}

\section{An Illustrative Example: The Rigid Body Control System}\label{sec:RigidBody}

Let us consider an example of rigid body attitude dynamics to discuss the theory developed in Section \ref{sec:MpcLie}. The rigid body attitude control system is defined by 
\begin{subequations}\label{rigid:original}
\begin{align}
\dot R &= R\hat \Omega, \\
\dot \Omega &= \MI^{-1} ( \MI \Omega \times \Omega) + \MI^{-1} u,
\end{align}
\end{subequations}
where \m{R \in \SO} (the set of \m{3\times3} real orthogonal matrices with determinant 1) is a rotation matrix that determines the attitude of the rigid body, \m{\Omega \in  \R^3}  defines the angular velocity of the rigid body; $u \in \R^3$ is the control torque;  $\mathbb I$ is the moment of inertial matrix of the rigid body; and  the hat map $\wedge$ maps  $\mathbb R^3$ vectors to $3\times 3 $ real skew symmetric matrices such that $\hat x y = x \times y$ for all $x,y \in \mathbb R^3$ with $\times$ as the vector product on $\R^3$. 

 Note that the manifold \m{ \SO \times \R^3 \subset \R^{3\times3}\times\R^3 } is an invariant set of the system dynamics \eqref{rigid:original}. To stably embed the system dynamics \eqref{rigid:original} into \m{\R^{3\times3}\times \R^3}, let us consider a function 
\begin{align}\label{eq:V}
W\times\R^3 \ni (X,\Omega) \mapsto  V(X,\Omega) \Let \frac{1}{4}\|X^\top X - I\|^2 \in \R,
\end{align}
where $W \Let \{ X \in \R^{3\times 3} \mid \det X >0\}$. Then, $V^{-1}(0) = \SO \times \mathbb R^3$ and 
\[
\nabla_X V(X,\Omega) = X(X^\top X - I), \quad \nabla_\Omega V(X,\Omega) = 0.
\]
It is easy to show that $V$ satisfies Assumption \ref{asm:V}, which will actually be proven in the proof of Theorem \ref{thm:rigidPrp}.
With the help of the function \m{V}, the system dynamics \eqref{rigid:original} is extended to the Euclidean space \m{\R^{3 \times 3} \times \R^3} as defined in \eqref{eq:DynEg} to be    
\begin{subequations}\label{rigid:tilde:eq}
\begin{align}
\dot X &= X\hat \Omega - \alpha X(X^TX - I), \label{R:s:eq}\\
\dot \Omega &= \MI^{-1} ( \MI \Omega \times \Omega) + \MI^{-1} u,
\label{Omega:s:eq}
\end{align}
\end{subequations}
where $(X,\Omega) \in \R^{3\times 3} \times \R^3$.

Take a reference trajectory 
\begin{align}\label{rigid:Ref}
\R \ni t \mapsto (R_0(t), \Omega_0 (t)) \in \SO \times \R^3
\end{align}
and the corresponding control signal $\R \in t \mapsto u_0(t)\in \R^3$ such that the trajectory obeys the system dynamics \eqref{rigid:original}. It is trivial to show that $R_0(t)$ satisfies Assumption \ref{asm:g0}.
Define the error trajectory as
\begin{align*}
\R \ni t & \mapsto \bigl(\ero(t), \erm(t)\bigr) \\
& \Let \bigl(X(t) R_0^{-1}(t) - I, \Omega(t) - \Omega_0(t)\bigr) \in \R^{3\times3}\times\R^3 
\end{align*}
such that \m{\ero = 0, \erm = 0} ensures that the system dynamics follows the reference trajectory. The linearized error dynamics along the reference state-control trajectory   $(R_0, \Omega_0,u_0) \in \SO \times \R^3\times\R^3$ is therefore given by
\begin{subequations}\label{rigid:linear}
\begin{align}\label{eq:LinearOrnt}
\dot \ero  &= R_0  \hat \erm  R_0^{-1} -2 \alpha \text{Sym} (\ero), \\
\dot \erm  &= \MI^{-1} ( \MI \erm \times \Omega_0+\MI  \Omega_0 \times  \erm) + \MI^{-1} \eru,
\end{align}
\end{subequations}
where \m{\text{Sym} (\ero)} is the symmetric component of the matrix \m{\ero} and $\eru(t)\Let u(t)-u_0(t)$.
Now we are in the position to split the error dynamics \eqref{eq:LinearOrnt} into the parallel and the transversal direction. The parallel direction is given by the Lie algebra \m{\so} (the set of \m{3 \times 3} skew symmetric matrices) of the Lie group \m{\SO} and the transversal direction is given by the perpendicular space \m{\prp{\so}} to the Lie algebra \m{\so} in \m{\R^{3\times3}} under the Euclidean norm, i.e., the set of \m{3 \times 3} symmetric matrices. 
Consequently, the attitude error dynamics \eqref{eq:LinearOrnt} is split into the parallel and the transversal direction, simplifying \eqref{rigid:linear} to  
\begin{subequations}\label{eq:LinearSplit}
\begin{align}\label{eq:OrntPrp}
\prp{\dot{\ero}} &= - 2 \alpha \prp{\ero},\\ \label{eq:OrntPrl} 
\prl{\dot{\ero}} &= R_0 \hat \erm R^{-1}_0,\\
\dot \erm  &= \MI^{-1} ( \MI \erm \times \Omega_0+\MI  \Omega_0 \times  \erm) + \MI^{-1} \eru,
\end{align}
\end{subequations}
where \m{\R \ni t \mapsto \prp{\ero}(t) \in \prp{\so}} and \m{\R \ni t \mapsto \prl{\ero}(t) \in \so}.  
\begin{remark}
It is worth noting that the parallel error dynamics \eqref{eq:OrntPrl} and the transversal error dynamics \eqref{eq:OrntPrp} are decoupled. Therefore, in the absence of a drift vector field, i.e., \m{\alpha=0}, the initial error in the transversal direction cannot be mitigated and that leads to a steady-state error in the transversal direction. In other words, for \m{\alpha=0}, the linearized error dynamics \eqref{eq:OrntPrp} cannot be stabilized to zero. 
\end{remark}

The discretized dynamics of \eqref{eq:LinearSplit}, by Euler's method, is given by
\begin{subequations}\label{rigid:linearDT}
\begin{align}
\dis[k+1]{\prp{\ero}} &= \dis{\prp{\ero}} - 2 \steplength \alpha \prp{\dis{\ero}}, \label{rigid:linearDT:perp}\\  
\dis[k+1]{\prl{\ero}} & = \dis{\prl{\ero}} +\steplength R_{0,k} \dis{\hat \erm} R^{-1}_{0,k}, \label{rigid:linearDT:prl}\\
\dis[k+1]{\erm}  &= \dis{\erm} + \steplength \MI^{-1} ( \MI \dis{\erm} \times \Omega_{0,k} + \MI  \Omega_{0,k} \times  \dis{\erm}) + \MI^{-1} \dis{\eru}, \label{rigid:linearDT:LieAlg}
\end{align}
\end{subequations}
where $h$ is the sampling time.
The following theorem proves exponential stability of \eqref{rigid:linearDT:perp}:
\begin{theorem}\label{thm:rigidPrp}
The transversal error dynamics \eqref{rigid:linearDT:perp} is exponentially stable if 
\begin{equation}\label{h:cond}
0< \alpha \steplength <1.
\end{equation}
 \begin{proof}
We employ Theorem \ref{thm:DisStabPrp} to establish the stability of the dynamics \eqref{rigid:linearDT:perp}. It is easy to prove that the function \m{V} in \eqref{eq:V} satisfies \m{V(XR)=V(X)} for all \m{X \in \R^{3\times3},} \m{R \in \SO} and \m{\nabla^2 V(I)(\prp{X},\prp{X}) = 2 \ip{\prp{X}}{\prp{X}}} for all \m{\prp{X} \in \prp{\so}}. Therefore the minimum eigenvalue \m{\lambda_{\text{min}}} and the maximum eigenvalue \m{\lambda_{\text{max}}} of the operator \m{\nabla^2V(I)} are equal to 2. Further, using the fact that \m{R^\top_0(t) R_0(t)= I} for all \m{t} leads to  \m{\glb=\gub=1}.   
Applying Theorem \ref{thm:DisStabPrp} to the dynamics \eqref{rigid:linearDT:perp}, we obtain \eqref{h:cond}.
\end{proof} 
\end{theorem}
\begin{remark}
In an identical manner, one can prove exponential stabilizability of the system dynamics \eqref{rigid:linearDT} by applying Theorem \ref{thm:DisStab}. 
\end{remark}
\subsection{Model predictive control design}

In this section, we design a model predictive tracking control of the discrete-time attitude control dynamics \eqref{rigid:linearDT}. 
Notice that the transversal dynamics \eqref{rigid:linearDT:perp} in \eqref{rigid:linearDT} is decoupled from \eqref{rigid:linearDT:prl} and \eqref{rigid:linearDT:LieAlg} and exponentially stable (see Theorem \ref{thm:rigidPrp}). Therefore, it is advantageous to choose \m{Q_{\ero},Q_{\ero}^f} in \eqref{ocp:cost} which decouples the cost \eqref{ocp:cost} along the parallel and the transversal direction, i.e.,
\m{\norm{\dis[N]{\ero}}_{Q^f_{\ero}}^2 = \norm{\dis[N]{\prl{\ero}}}_{Q^f_{\ero}}^2 +  \norm{\dis[N]{\prp{\ero}}}_{Q^f_{\ero}}^2 } and \m{\norm{\dis[i]{\ero}}_{Q_{\ero}}^2 =  \norm{\dis[i]{\prp{\ero}}}_{Q_{\ero}}^2 +  \norm{\dis[i]{\prl{\ero}}}_{Q_{\ero}}^2} for all \m{k=0,\ldots,N-1,} so that we can ignore the transversal dynamics as it is not influencing the optimization problem. 
 
	Consequently, an \m{N} horizon optimal control problem \eqref{ocp:DynDis} at a given time instant \m{k} for the system dynamics \eqref{rigid:linearDT} with the performance objective \eqref{ocp:cost} and constraints \eqref{ocp:constraints} is given by
\begin{equation}\label{rigid:ocp}
\begin{aligned}
&\minimize_{\{\dis[k+i|k]{\eru}\}_{i=0}^{N-1}} \J \big(\dis[k:k+N|k]{\prl{\ero}},\dis[k:k+N|k]{\erm},\dis[k:k+N|k]{\eru}\big)\\
&\text{subject to}\\ 
&\quad
\begin{cases}
\begin{cases} 
\dis[k+i+1|k]{\prl{\ero}} = \dis[k+i|k]{\prl{\ero}} +\steplength R_{0,k+i|k} \dis[k+i|k]{\hat \erm} R^{-1}_{0,k+i|k}\\
\dis[k+i+1|k]{\erm} = \dis[k+i|k]{\erm} + \steplength \MI^{-1} ( \MI \dis[k+i|k]{\erm} \times \Omega_{0,k+i|k})\\
+ \steplength \MI^{-1} (\MI  \Omega_{0,k+i|k} \times  \dis[k+i|k]{\erm}) + \MI^{-1} \dis[k+i|k]{\eru}\\
\dis[k+i|k]{\eru} \in \ceru[k+i] 
\end{cases}\\ 
\text{for all }\; i=0,\ldots,N-1, \\
\begin{cases}
\dis[k+i|k]{\prl{\ero}} \in \cero[k+i]\\ 
\dis[k+i|k]{\erm} \in \cerm[k+i] 
\end{cases}
\quad \text{for all }\; i=1,\ldots,N, \\
\big(\dis[k|k]{\prl{\ero}},\dis[k|k]{\erm}\big) = \big(\dis[k]{\prl{\ero}},\dis[k]{\erm} \big)
\end{cases} 
\end{aligned} 
\end{equation}
where \m{\big(\dis[k]{\prl{\ero}},\dis[k]{\erm} \big)} is fixed, and $\cero[k+i]$ and $\cerm[k+i] $ are admissible sets for $\dis[k+i|k]{\prl{\ero}} $  and $\dis[k+i|k]{\erm}$, respectively. The quadratic program \eqref{rigid:ocp} can be solved in MATLAB using an optimization modeling toolbox YALMIP \cite{yalmip}. A detailed exposition of computational complexity of real-time MPC exists in \cite{richter,rubagotti,Valentin}. The optimal control problem \eqref{rigid:ocp} is solved at each time instant \m{k} and the control 
\[
u \Let \Rtraj{u} + \eru,
\]
where \m{\eru \Let \dis[k|k]{\eru}} for \m{[kh,(k+1)h{[},} is applied to the system as shown in Figure \ref{fig_MPC}.   

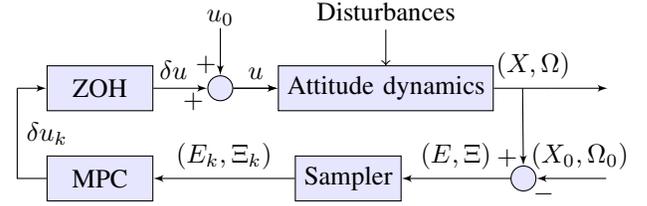
\begin{figure}[htb]
\centering
\tikzstyle{block} = [draw, fill=blue!10, rectangle, 
    minimum height=1.7em, minimum width=4em]
\tikzstyle{sum} = [draw, fill=blue!10, circle, node distance=0.8cm]
\tikzstyle{input} = [coordinate]
\tikzstyle{output} = [coordinate]
\tikzstyle{pinstyle} = [pin edge={to-,thin,black}]

% The block diagram code is probably more verbose than necessary
\begin{tikzpicture}[auto, node distance=1.6cm,>=latex']
    % We start by placing the blocks
    \node [input, name=control_input] {};
    \node [sum, below of=control_input] (control_sum) {};
    \node [block, left of=control_sum] (zoh) {ZOH};
    \node [block, right of=control_sum, , pin={[pinstyle]above:Disturbances},
           node distance=2.2cm] (system) {Attitude dynamics};
    \node [block, below of=system,node distance=1.2cm,xshift=-0.5cm] (sampler) {Sampler};
    \node [block, left of=sampler, node distance=3.3cm] (controller) {MPC};
    \node [output, right of=system, node distance=2.2cm] (output) {};
    \node [sum, right of=sampler, node distance=2.33cm ] (state_sum) {};
    \node [input, right of =state_sum](state_input) {};
    \node [input, left of =zoh, node distance=1.1cm](dummy_input) {};
    \node [input, right of=state_sum,node distance=1.1cm](state_input) {};
    \node [input, right of=output,node distance=0.75cm](dummy_output) {};
    % We draw an edge between the controller and system block to 
    % calculate the coordinate u. We need it to place the measurement block.
    \draw [<-] (controller) --node{\m{(\dis{\ero},\dis{\erm})}} (sampler); 
    \draw [->] (zoh)node[above right,xshift=19pt]{\m{\eru}} -- node[below,yshift=2pt,xshift=5pt]{$+$}(control_sum);
    \draw [->] (control_sum) -- node[name=u] {$u$} (system);
    \draw [<-] (sampler) -- node{\m{(\ero,\erm)}} (state_sum);
    \draw [->] (state_input) node[above,xshift=-10pt]{$(\Rtraj{X},\Rtraj{\Omega})$} -- node[below left, yshift=1pt,xshift=-3pt]{$-$}(state_sum);
    \draw [->] (output) -- node[above left]{$(X,\Omega)$}(dummy_output);
    % Once the nodes are placed, connecting them is easy. 
    \draw [draw,->] (control_input) node[yshift=4pt]{$\Rtraj{u}$}--node[left,xshift=2pt,yshift=-5pt]{$+$}(control_sum);
    \draw [->] (control_sum) -- (system);
    \draw [-] (system) -- node [name=y]{} (output);
    \draw [->] (y) -- node[below left, yshift=-5pt,xshift=1.5pt]{$+$}(state_sum);
    \draw [-] (controller) -| node [right,near end] {$\dis{\eru}$} (dummy_input);
    \draw [->] (dummy_input) -- (zoh);
\end{tikzpicture}
     \caption{The sampled-data closed-loop system with MPC.}\label{fig_MPC}
 \end{figure}

\section{Simulation Results}\label{sec:Simulations}
We simulate our MPC in a sampled-data system as shown in Figure \ref{fig_MPC}.
The finite dimensional quadratic programming problem \eqref{rigid:ocp} is solved at each time instant \m{k} to calculate a feedback control law. The moment of inertia matrix of the rigid body in \eqref{rigid:ocp} is   
\[ \MI=\text{diag}(4.250, 4.337, 3.664),\] 
which was taken from a satellite from the European Student Earth Orbiter (ESEO) \cite{Hegrenas}.
Let us track a reference trajectory 
\[
\R \ni t \mapsto (R_0(t), \Omega_0 (t)) \in \SO \times \R^3,
\]
where
\[
R_0(t) \Let \exp\left(t \hat{e_1} \right)\exp\left(t \hat{e_2} \right)\exp\left(t \hat{e_3} \right) 
\]
with $e_i$ as the unit vector along the $i$th axis, and  
\[
\Omega_0 (t) \Let  \left(1+\cos t,  \sin t - \sin t \cos t,  \cos t + \sin^2 t\right)^\top,
\]
using the MPC control law with the corresponding reference control signal 
\[
\R \ni t \mapsto u_0(t) = \mathbb I \dot \Omega_0 (t) - (\mathbb I \Omega_0 (t))\times \Omega_0 (t) \in \R^3.
\]

	The initial data and parameters considered for the optimization \eqref{rigid:ocp} are the following:
\begin{itemize}
\item \m{\dis[0]{\prl{\ero}} = \prl{\big(R_0(0.2)-R_0(0)\big)}, \quad \dis[0]{\erm} = (0,0,0)^\top,}% \exp (0.99\pi \hat e_2), 
\item \m{Q_{\ero}=100 \gI, \quad Q_{\erm}=10 \gI, \quad Q_{\eru}= 0.01 \gI},
\item \m{Q^f_{\ero} = 100 \gI, \quad Q^f_{\erm} = 10 \gI},
\item  sampling time: \m{h = 0.2} \si{\sec},
\item MPC time horizon: \m{N = 4}.
\end{itemize}
The MPC controller takes the tracking error measurements \m{\big(\dis{\ero},\dis{\erm}\big)} for computing the feedback control \m{\dis{\eru}} at the each iteration \m{k}. These tracking error measurements are calculated as a difference of the reference states \m{\big(R_{0,k},\Omega_{0,k} \big)} and the states obtained from the ODE simulation (we have used the MATLAB integrator, ode45, with the options, $RelTol = AbsTol = 10^{-6}$) of the rigid body dynamics \eqref{rigid:original} . In turn, the ODE simulation is driven by the zero-order hold control actions generated by the MPC controller, and that forms the closed-loop MPC system; see Figure \ref{fig_MPC}.
We simulate three different scenarios as follows:

\subsection{Case 1:  Loose constraint and no noise}\label{subsec:case1} 
In the first case study, we consider the following state and control constraints \eqref{ocp:constraints}: 
 \[ 
\cero=\so,\quad \cerm=\R^{3},\quad \ceru= {U} - \dRtraj{u} \quad \text{for each } k, 
 \]
where \m{U \Let \{y \in \R^3 | -10 \leq y_i \leq 10 \; \text{ for } i = 1,2,3 \} }.
The closed-loop system with the designed MPC shows a successful tracking performance as the error trajectory \m{(\ero (t),\erm (t))} tends to zero quickly (see Figure \ref{figure:Error1}), and the optimal control profile obeys the control constraints as shown in  Figure \ref{figure:Control1}. It is worth noting that the angular velocity error $ \erm $ shown in Figure \ref{figure:Error1} increases initially from zero in order to mitigate the initial orientation error $\ero.$   
 
\begin{figure}[htb]
 \centering
\psfrag{Time}[c][c][1][0]{Time [\si{\sec}]}
 \includegraphics[width=1\columnwidth]{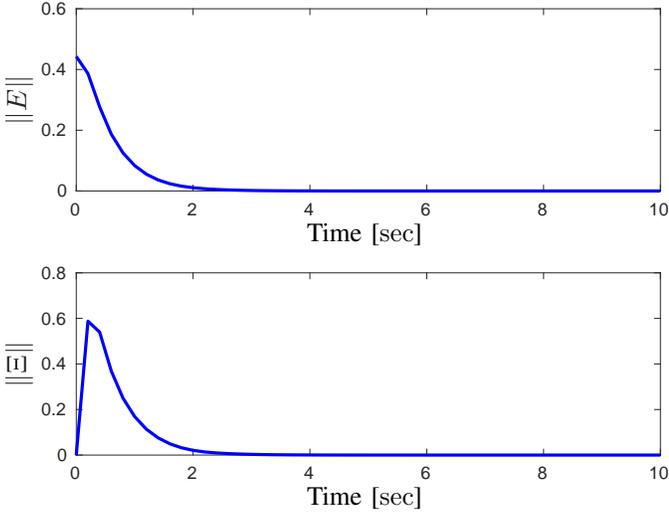}
 \caption{ The tracking errors  for Case 1. }
\label{figure:Error1}
\end{figure}

\begin{figure}[htb]
 \centering
\psfrag{Time}[c][c][1][0]{Time [\si{\sec}]}
\psfrag{Time }[c][c][1][0]{Time [\si{\sec}]}
\includegraphics[width=1\columnwidth]{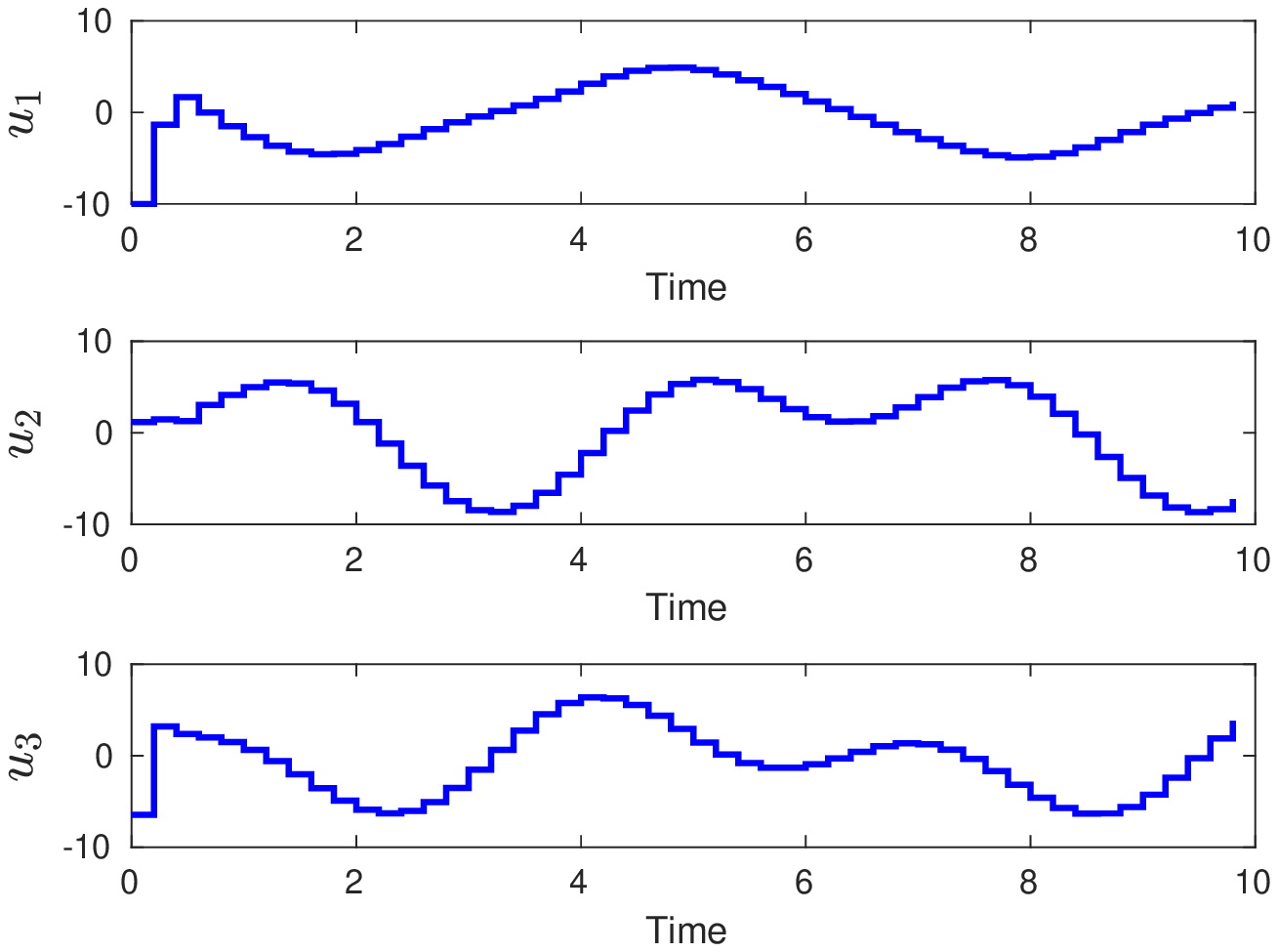}
 \caption{ Zero-order hold control for the sampled data system with MPC for Case 1. }
\label{figure:Control1}
\end{figure}

\subsection{Case 2: Tight control constraint} 
The second case we study is the one with the following tight constraints:
 \[ 
\cero=\so,\quad \cerm=\R^{3},\quad \ceru= {U} - \dRtraj{u} \quad \text{for each } k, 
 \]
where \m{{U} \Let \{y \in \R^3 | -6 \leq y_i \leq 6 \; \text{ for } i = 1,2,3 \} }.
The closed-loop system with the designed MPC considering a tight control bound  shows a compromised tracking performance. As the control trajectory hits the control bounds (see Figure \ref{figure:Control2}), the error trajectory \m{(\ero (t),\erm (t))} deviates from zero, as shown in Figure \ref{figure:Error2}.
 
 \begin{figure}[htb]
  \centering
\psfrag{Time}[c][c][1][0]{Time [\si{\sec}]}
    \includegraphics[width=1\columnwidth]{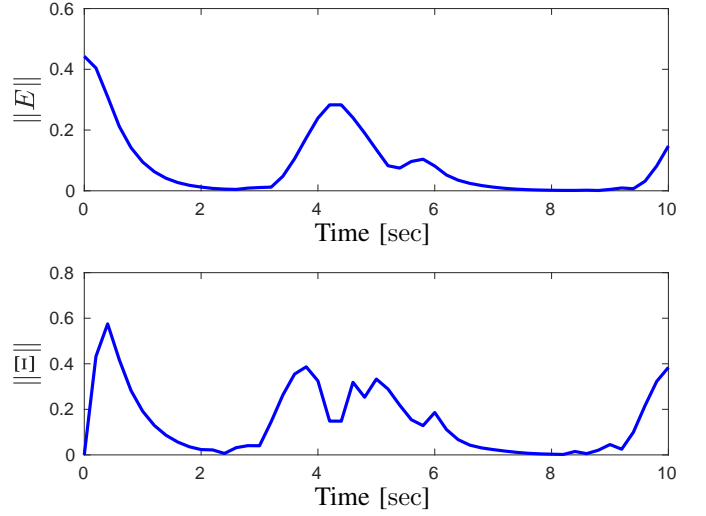}
\caption{The tracking errors for Case 2.}
    \label{figure:Error2}
\end{figure}
 
 \begin{figure}[htb]
  \centering
\psfrag{Time}[c][c][1][0]{Time [\si{\sec}]}
    \includegraphics[width=1\columnwidth]{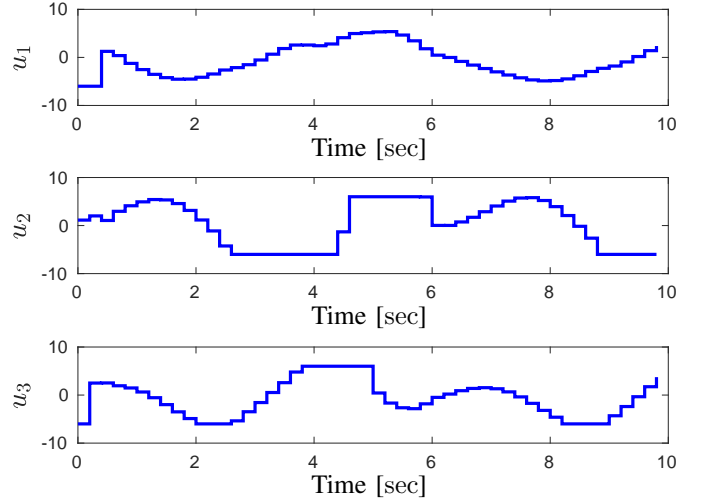}
    \caption{ Zero-order hold control for the sampled data system with MPC  for Case 2.}
\label{figure:Control2}
 \end{figure}

\subsection{Case 3: Noisy measurements} 
In this case study, the state and control constraints are considered as in Case~1 in Section \ref{subsec:case1}; however, a measurement noise in the tracking error  
\[
\big(\dis[k+i+1|k]{\prl{\ero}}, \dis[k+i+1|k]{\erm}\big) \in \so \times \R^3 \quad \text{for} \quad i=1,\ldots,N, 
\]
is realized by the independent and identically distributed (i.i.d.) random variables, i.e.,  $w \sim \mathcal{N}(0, \sigma_w^2),$ where $\sigma_w=0.03$.
Due to the noisy state measurements, the error trajectory \m{(\ero (t),\erm (t))} fluctuates around zero instead of stabilizing at zero; see Figure \ref{figure:Error3}. However, the tracking performance of the closed loop system is similar to Case 1, as shown in Figure \ref{figure:Error3} and Figure \ref{figure:Control3}.

\begin{figure}[htb]
  \centering
\psfrag{Time}[c][c][1][0]{Time [\si{\sec}]}
    \includegraphics[width=1\columnwidth]{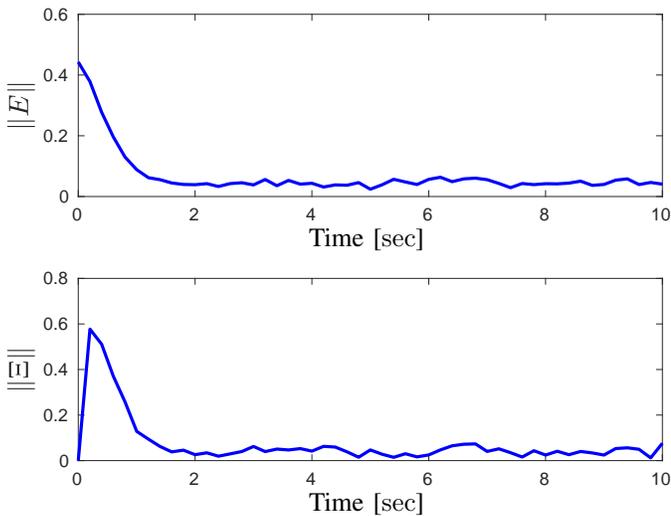}
    \caption{The tracking errors for Case 3.}
    \label{figure:Error3}
\end{figure}

\begin{figure}[htb]
  \centering
\psfrag{Time}[c][c][1][0]{Time [\si{\sec}]}
    \includegraphics[width=1\columnwidth]{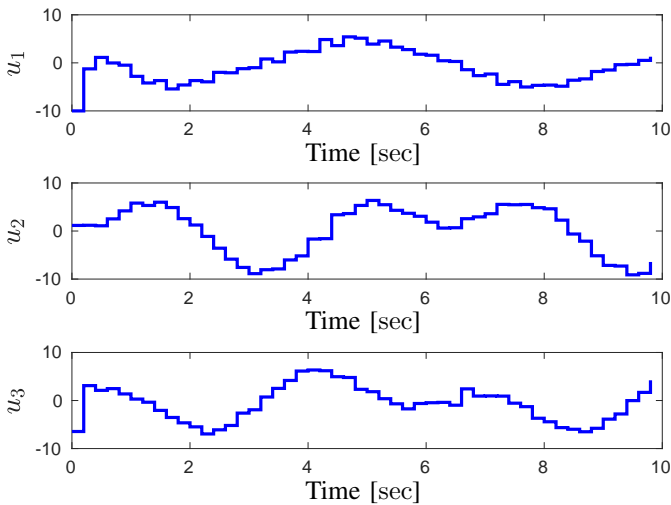}
    \caption{ Zero-order hold control for the sampled data system with MPC  for Case 3. }
\label{figure:Control3}
 \end{figure}

\section{Conclusion}\label{sec:Conclusion}
In this paper, we have presented a technique  to design model predictive tracking controllers for control systems evolving on manifolds in Euclidean spaces. We have applied the proposed technique to the systems on matrix Lie groups and demonstrated its potency by designing a linear MPC law for the rigid body attitude dynamics. 
Our approach simplifies MPC design for control systems on manifolds. This development could be quite useful for control engineers in dealing with nonlinear mechanical control system applications in practice.

\bibliographystyle{IEEEtran}
\bibliography{TAC_MPC_on_Manifolds}

\end{document}